\documentclass[12pt,reqno]{article}
\usepackage{amsmath,amsthm,amsfonts,amssymb,amscd,amsbsy}
\usepackage{graphics}

\usepackage{color}

\usepackage{stmaryrd}

\begin{document}

\theoremstyle{plain}
\newtheorem{thm}{Theorem}[section]
\newtheorem{theorem}[thm]{Theorem}
\newtheorem{main theorem}[thm]{Main Theorem}
\newtheorem{lemma}[thm]{Lemma}
\newtheorem{corollary}[thm]{Corollary}
\newtheorem{proposition}[thm]{Proposition}
\newtheorem{claim}[thm]{Claim}

\theoremstyle{definition}
\newtheorem{notation}[thm]{Notation}

\newtheorem{remark}[thm]{Remark}
\newtheorem{remarks}[thm]{Remarks}
\newtheorem{conjecture}[thm]{Conjecture}
\newtheorem{definition}[thm]{Definition}
\newtheorem{example}[thm]{Example}

\newcommand{\Max}{{\rm Max \ }}
\newcommand{\sA}{{\mathcal A}}
\newcommand{\sB}{{\mathcal B}}
\newcommand{\sC}{{\mathcal C}}
\newcommand{\sD}{{\mathcal D}}
\newcommand{\sE}{{\mathcal E}}
\newcommand{\sF}{{\mathcal F}}
\newcommand{\sG}{{\mathcal G}}
\newcommand{\sH}{{\mathcal H}}
\newcommand{\sI}{{\mathcal I}}
\newcommand{\sJ}{{\mathcal J}}
\newcommand{\sK}{{\mathcal K}}
\newcommand{\sL}{{\mathcal L}}
\newcommand{\sM}{{\mathcal M}}
\newcommand{\sN}{{\mathcal N}}
\newcommand{\sO}{{\mathcal O}}
\newcommand{\sP}{{\mathcal P}}
\newcommand{\sQ}{{\mathcal Q}}
\newcommand{\sR}{{\mathcal R}}
\newcommand{\sS}{{\mathcal S}}
\newcommand{\sT}{{\mathcal T}}
\newcommand{\sU}{{\mathcal U}}
\newcommand{\sV}{{\mathcal V}}
\newcommand{\sW}{{\mathcal W}}
\newcommand{\sX}{{\mathcal X}}
\newcommand{\sY}{{\mathcal Y}}
\newcommand{\sZ}{{\mathcal Z}}
% Sonderbuchstaben mit Doppellinie
\newcommand{\A}{{\mathbb A}}
\newcommand{\B}{{\mathbb B}}
\newcommand{\C}{{\mathbb C}}
\newcommand{\D}{{\mathbb D}}
\newcommand{\E}{{\mathbb E}}
\newcommand{\F}{{\mathbb F}}
\newcommand{\G}{{\mathbb G}}
\newcommand{\HH}{{\mathbb H}}
\newcommand{\I}{{\mathbb I}}
\newcommand{\J}{{\mathbb J}}
\newcommand{\M}{{\mathbb M}}
\newcommand{\N}{{\mathbb N}}
\renewcommand{\P}{{\mathbb P}}
\newcommand{\Q}{{\mathbb Q}}
\newcommand{\re}{{\mathbb R}}
\newcommand{\R}{{\mathbb R}}
\newcommand{\T}{{\mathbb T}}
\newcommand{\U}{{\mathbb U}}
\newcommand{\V}{{\mathbb V}}
\newcommand{\W}{{\mathbb W}}
\newcommand{\X}{{\mathbb X}}
\newcommand{\Y}{{\mathbb Y}}
\newcommand{\Z}{{\mathbb Z}}
\newcommand{\la}{{\lambda}}
\newcommand{\al}{{\alpha}}
\newcommand{\be}{{\beta}}

\newcommand{\ve}{{\varepsilon}}
\newcommand{\vr}{{\varphi}}
\newcommand{\ab}{|}
\def\ga{{\gamma}}
\def\vr{{\varphi}}
\def\la{{\lambda}}

\def\al{{\alpha}}
\def\be{{\beta}}
\def\ve{{\varepsilon}}
\def\lv{\left\vert}
\def\rv{\right\vert}

\def\sen{\operatorname{{sen}}}
\def\tr{\operatorname{{tr}}}

\def\re{{\Bbb{R}}}
\def\bc{{\mathbb C }}
\def\fT{{\frak T}}
\def\nb{{\mathbb N}}
\def\bz{{\mathbb Z}}

\def\pc{{\mathcal P}}

\centerline{\huge\bf An algorithm for the word entropy}

\vskip .3in 

\centerline{\sc S\'ebastien Ferenczi}

\centerline{\sl IMPA - CNRS UMI 2924}

\centerline{\sl Estrada Dona Castorina 110, 22460-320 Rio de Janeiro, RJ, Brasil}

\centerline{\sl and}

\centerline{\sl Institut de Math\'ematiques de Marseille, UMR 7373 CNRS,}

\centerline{\sl 163, avenue de Luminy, 13288 Marseille Cedex 9, France}

\vskip .2in

\centerline{\sc Christian Mauduit}

\centerline{\sl Universit\'e d'Aix-Marseille and Institut Universitaire de France,}

\centerline{\sl Institut de Math\'ematiques de Marseille, UMR 7373 CNRS,}

\centerline{\sl 163, avenue de Luminy, 13288 Marseille Cedex 9, France}

\vskip .2in

\centerline{\sc Carlos Gustavo Moreira}

\centerline{\sl Instituto de Matem\'atica Pura e Aplicada,}

\centerline{\sl Estrada Dona Castorina 110,}

\centerline{\sl 22460-320 Rio de Janeiro, RJ, Brasil}

\vskip .3in

{\bf Abstract:}
For any infinite word $w$ on a finite alphabet $A$, the complexity function $p_w$ of $w$ is the sequence counting, for each non-negative $n$, the number $p_w(n)$ of words of length $n$ on the alphabet $A$ that are factors of the infinite word $w$ and the the entropy of $w$ is the quantity $E(w)=\lim\limits_{n\to\infty}\frac 1n\log p_w(n)$.
For any given function $f$ with exponential growth, Mauduit and Moreira introduced in [MM17]  the notion of word entropy $E_W(f) = \sup \{E(w), w \in \A^{\N}, p_w \le f \}$
%associated to the function $f$ as the supremum of the entropy of all the infinite words $w$ with complexity function bounded by $f$
and showed its links with
fractal dimensions of sets of infinite sequences with complexity function bounded by $f$.
The goal of this work is to give an algorithm to estimate with arbitrary precision $E_W(f)$ from finitely many values of $f$.

\vskip .1in



This work was supported by CNPq, FAPERJ and the Agence Nationale de la Recherche project ANR-14-CE34-0009 MUDERA.
\vskip .3in

\section{Introduction}

This work concerns the little-explored field of word combinatorics in positive entropy, which means the study of
infinite words on a finite alphabet with a
complexity function (see Definition \ref {def1.2}) of exponential growth.
There are not many results on this topic, besides the well-known one of Grillenberger \cite{Gri} who built symbolic systems of any given entropy. 

Mauduit and Moreira introduced in [MM17] new notions in this context
with the arithmetic motivation to study sets of numbers from the interval $[0, 1]$ whose expansion (in a given base $q$) has a complexity function bounded by a given function $f$.
The determination of the Hausdorff dimension of these sets gave rise to a new quantity $E_W(f)$, called {\em word entropy} of $f$, which turns to be equal to the topological entropy of the
shift on the set of corresponding expansions.  

The computation of $E_W(f)$ is trivial when $E_0(f)$, the {\em exponential growth rate} of $f$ (defined in (\ref {E_0})), is equal to zero or if $f$ is itself a complexity function.
Otherwise, results can be surprising, even when $f$ is very regular: for example, in \cite{MM17} it is shown that for the function $f$ defined
for any non-negative integer $n$ by $f(n)=\lceil \frac32\rceil^n$, we have $E_W(f) = \log (\frac {1 + \sqrt 5}2)$. 
Another striking result (see Theorem 2.9 from  [MM18]) says that if $f$ verifies the quite natural conditions $(\mathcal C^*)$ (see Definition \ref{defC^*}), then the ratio
$E_W(f) / E_0(f)$  lies always in the interval $]\frac 12, 1 ]$ and moreover we have  $$\inf \{ \frac {E_W(f)} {E_0(f}), f \mbox{  $satisfies$  } (\mathcal C^*) \} = \frac 12.$$
Indeed, in the overwhelming majority of  cases, we do not have access to an exact value of the word entropy.
Thus in this work we propose an algorithm to get an approximate value of the word entropy, using in depth the combinatorial properties of the symbolic system.

\section{Definitions and notations}

We denote
by $q$ a fixed integer greater or equal to $2$, by $A$ the finite
alphabet $A=\{0,1,\dots,q-1\}$, by $A^*=\bigcup\limits_{k\ge0} A^k$
the set of finite words on the alphabet $A$ and by $A^{\N}$ the set
of infinite words (or infinite sequences of letters) on the alphabet
$A$.
More generally, if $\Sigma \subset A^*$, we denote by  $\Sigma ^ \N$ the set of infinite words obtained by concatenating elements of $\Sigma$.
If $w\in A^{\N}$ we denote by $L(w)$ the set of finite factors of $w$:
$$
L(w)=\{v\in A^*,\,\, \exists \, (v',v'')\in A^*\times A^{\N}, \, w=v'v v''\}
$$
and, for any non-negative integer $n$, we write $L_n(w)=L(w)\cap A^n$.
For any $Y \subset A^\N$ and $n \in \N$ we denote $L_n(Y) = \bigcup\limits_{w\in Y} L_n(w)$.
%For any $n \in \N$ we say that $v \in L_n(w)$ is a special factor (of length $n$) of $w$ if $v$ is prefix of more than one word in $L_{n+1}(w)$.
If $w\in A^n, n \in \N$ we denote $|w| = n$ the length of the word $w$ and if $S$ is a finite set, we denote by $|S|$ the number of elements of $S$.
For any $(a, b) \in \R^2$ with $a \le b$, we denote by $\llbracket a, b \rrbracket$ the set $[a, b] \cap \Z$ and for any
$x$ real number, we denote
$
\lfloor x\rfloor = \max\{n\in\Z, n\le x\},
\lceil x\rceil=\min\{n\in \Z, x\le n\}
$
and
$\{ x \} = x - \lfloor x\rfloor$.

%We will use at several stage the following classical lemma concerning sub-additive sequences due to Fekete [Fek23]:

%\begin{lemma}\label{lemFekete}
%If $(a_n)_{n\ge1}$ is a sequence of real numbers such that $a_{n+n'} \le a_n + a_{n' }$ for any positive integers $n$ and $n'$, then the sequence
%$\left( \frac {a_n} n\right)_{n\ge1}$ converges to $\inf_{n\ge 1}\frac {a_n} n$.
%\end{lemma}

Let us recall the following classical lemma concerning sub-additive sequences due to Fekete [Fek23]:

\begin{lemma}\label{lemFekete}
If $(a_n)_{n\ge1}$ is a sequence of real numbers such that $a_{n+n'} \le a_n + a_{n' }$ for any positive integers $n$ and $n'$, then the sequence
$\left( \frac {a_n} n\right)_{n\ge1}$ converges to $\inf_{n\ge 1}\frac {a_n} n$.
\end{lemma}

\begin{definition}\label{def1.2}
 The { \it complexity function} of $w\in A^{\N}$ is defined for any non-negative integer
  $n$ by $p_w(n)=|L_n(w)|$. % and the  { \it entropy} of $w$ is $E(w)=\lim\limits_{n\to\infty}\frac 1n\log p_w(n)$.

\end{definition}

For any $w\in A^{\N}$ and for any $(n,n')\in \N^2$ we have
  $L_{n+n'}(w)\subset L_n(w) L_{n'}(w)$ so that 
  $p_w(n+n')\le p_w(n) p_w(n')$
and it follows from Lemma \ref{lemFekete}
that for any
$w\in A^{\N}$, the sequence $\left( \frac 1n \log p_w(n)\right)_{n\ge1}$ converges
 to $\inf_{n\ge 1} \frac 1n \log p_w(n)$.
We denote
$$E(w)=\lim\limits_{n\to\infty}\frac 1n\log p_w(n) = h_{top} (X(w), T)$$
%can be shown (see for example [K\accent23ur03]) that $E(w)$ is equal to $h_{top} (X(w), T)$, 
the topological entropy of the symbolic dynamical system $(X(w),T)$ where
$T$ is the one-sided shift on $A^{\N}$ and $X=\overline{orb_T(w)}$ is
the closure of the orbit of $w$ under the action of $T$ in $A^{\N}$
($A^{\N}$ is equipped with the product topology of the discrete
topology on $A$, i.e. the topology induced by the distance
$d(w,w')= \exp(-\text{min} \{n \in \N | \, w_n \ne w_n' \})$).

The complexity function gives information about the statistical
properties of an infinite sequence of letters. In this sense, it
constitutes one possible way to measure the random behaviour of an
infinite sequence:
see [Que87, Fer99, PF02].
%This work concerns the little-explored field of combinatorics of words with exponential complexity (or positive entropy).
%and see [MS97, MS98] for connections
%between measure of normality and other measures of pseudorandomness.
%Obviously, we have $1 \le p_w(n) \le q^n$ for any non-negative integer $n$
%and it is easy to prove the following lemma:

%\begin{lemma}\label{lem1.2}
 % For any $w\in A^{\N}$ and for any $(n,n')\in \N^2$ we have
 % $L_{n+n'}(w)\subset L_n(w) L_{n'}(w)$ and $p_w(n+n')\le p_w(n) p_w(n')$.
%\end{lemma}

%\noindent
%Lemmas \ref{lemFekete}  and \ref{lem1.2} have two important consequences that we will use several times:

%\medskip

%\noindent
%{\bf Consequence 1:} For any
%$w\in A^{\N}$, the sequence $\left( \frac 1n \log
 % p_w(n)\right)_{n\ge1}$ converges
 %to \break $\inf_{n\ge 1} \frac 1n \log p_w(n)$.

%\medskip

%\noindent
%{\bf Consequence 2:} If there exist a real number $\lambda_0$ ($1 \leq \lambda_0 \leq q$) and an integer $n_0$ such that $p_w(n_0)<\lambda_0^{n_0}$,
%then $p_w(n)=O(\lambda^n)$ for some $\lambda < \lambda_0$.

\section{Exponential rate of growth and word entropy of a function}

%\begin{definition}\label{def2.1}

%We say that a function $f$ from $\N$ to $\re^+$ verifies the conditions $(\mathcal C)$ if

%(i) the sequence $\left( f(n)\right)_{n\ge1}$ is strictly increasing;

%(ii) $\exists\,\, n_0\in\N$, $\forall n\ge n_0 \Rightarrow f(2n)\le
%f(n)^2 \text{ and }f(n+1) \le f(1) f(n) $;

%(iii) the sequence $\left( \frac 1n \log f(n)\right)_{n\ge1}$ converges.

%\end{definition}

For any given function $f$ from $\N$ to $\R^{+}$, we denote
$$W(f)=\{w\in A^{\N}, p_w(n)\le f(n), \forall n \in \N\},$$
$$\sL_n(f)=\bigcup\limits_{w\in W(f)} L_n(w)$$ 
and $E_0(f)$ the limiting lower exponential growth rate of $f$
\begin{equation} \label{E_0}
E_0(f)=\lim\limits_{n\to\infty} \inf \frac 1n \log f(n).
\end{equation}
For any $(n,n')\in \N^2$ we have
$\sL_{n+n'}(f)\subset \sL_n(f) \sL_{n'}(f)$ so that %we can deduce from Lemma \ref{lemFekete} applied to $a_n = \log |\sL_n(f)|$ that 
the sequence
$\left( \frac 1n \log |\sL_n(f)| \right)_{n\ge1}$ converges
 to $\inf_{n\ge 1} \frac 1n \log |\sL_n(f)|$,
 which is the topological entropy of the subshift $(W(f), T)$ :
 $$h_{top} (W(f), T) = \lim_{n\to+\infty}\frac 1n\log |\sL_n(f)| = \inf_{n\ge 1} \frac 1n \log |\sL_n(f)|.$$
The notion of w-entropy (or word-entropy) of $f$ is defined in [MM17] as follow :

\begin{definition}\label{defg}
If $f$ is a function from $\N$ to $\R^{+}$, the
{\it w-entropy} (or  {\it word entropy}) of $f$ is the quantity
$$E_W (f) =\sup_{\substack{w \in W(f)}} E(w).$$
\end{definition}

The papers [MM10] and [MM12] concern the case $E_0(f)=0$ and [MM17] the case of positive entropy. In particular 
the word entropy of $f$ is equal to the topological entropy of the subshift
$(W(f), T)$ (see Theorem 2.3 from [MM17]):
for any function $f$ from $\N$ to $\R^{+}$, we have
$$E_W(f) = \lim_{n\to+\infty}\frac 1n\log (|\sL_n(f)|) = h_{top} (W(f), T) .$$
(see also beginning of Section 4 from [MM17] and Chapter 8 from [Wal82] to understand this result as a consequence of the variational principle).

The word entropy of $f$ allows to compute exactly the fractal dimensions of the set
%\begin{equation} \label {C(f)}
%C(f)=\{x = \sum \limits_{i \ge 0} \frac{w_i}{q^{i+ 1}} \in [0,1] , w(x) = {w_0}{w_1}\cdots{w_i}\cdots \in W(f)\}
%\end {equation}
of real numbers from the interval $[0,1]$ the $q-$adic expansion of which has a
complexity function bounded by $f$.
(see Theorem 5.1 from [MM17]).
Note that several authors have applied the notion of dimension introduced by Hausdorff in [Hau19]
to number theoretical problems (see [Bug04, Chapters V and VI] for a very good survey on these questions and
[Fal90, Chapters 2 and 3] for basic definitions concerning fractal dimensions).

%The second part of the paper is devoted to
%the study of the properties of the w-entropy $E_W$ and its relations with the exponential rate of growth $E_0$.
%Infinite words whose complexity function has an exponential grow but low initial values play a special role in this study and we define the following class of infinite words
%on the alphabet $\{0, 1 \}$:

%\begin{definition}\label{pre-sturmian}
%We say that $w \in \{0, 1 \}^\N $ is a  {\it pre-sturmian} infinite word of order $k$ if $w$ is not ultimately periodic and if, for any non-negative integer $n \le k$, we have $p_w (n) = n + 1$.
%\end{definition}

%\begin{remark}
%It follows from a classical result due to Coven and Hedlund ([CH73]) that non ultimately periodic infinite words $w$
%with lowest possible complexity function $p_w$ are the ones for which $p_w(n)=n+1$ for any non-negative
%integer $n$. Such infinite words, called sturmian words, have been
%extensively studied since their introduction by Hedlund and
%Morse in [HM38] and [HM40] (see [Lot02, chapter 2] and [PF02, Chapter 6] for very good surveys on sturmian words).
%It follows from Definition \ref {pre-sturmian} that a  sturmian word is a pre-sturmian word of any order.
%\end{remark}

\begin{definition}\label{defC^*}

We say that a function $f$ from $\N$ to $\R^{+}$ satisfies the conditions $(\mathcal C^*)$ if

i) for any $n \in \N$ we have $f(n+1) > f(n) \ge n+1$ ;

ii) for any $(n, n' ) \in \N^2$ we have $f(n+n') \le f(n) f(n') $.
\end{definition}

For any function $f$ from $\N$ to $\R^{+}$ we have $E_W(f) \le E_0(f)$ and it is easy to give examples of function $f$ for which the entropy ratio 
$E_W(f)/E_0(f)$ can be made arbitrarily small (see beginning of Section 7 from [MM17]).
When $f$ satisfies the quite natural conditions $(\mathcal C^*)$, it still might happen that $E_W(f) < E_0(f)$ (see Sections 7.2 and 7.4 from [MM17]), but Mauduit and Moreira proved
the following theorem  (see Theorem 4.2 and Remark 4.3 from
[MM18]):
%that  $\inf \{ \rho (f), f \mbox{  $satisfies$  } (\mathcal C^*) \} = \frac 12.$
\begin{theorem} \label{1/2theorem}
%$\,$ \hfill \break
If $f$ is a function from $\N$ to $\R^{+}$ satisfying the conditions $(\mathcal C^*)$, then $E_W(f)> \frac 12 E_0(f)$.
\end{theorem}
\noindent
Moreover, the constant $\frac 12$ in Theorem \ref {1/2theorem} is optimal (see Theorem 5.1 from [MM18]).

As mentioned in [MM17] it is in general more difficult to compute $E_W(f)$ than $E_0(f)$.
%so that it would be interesting to provide an algorithm allowing to estimate with 
%an arbitrary precision.
The goal of this work is to give
an algorithm which allows us to estimate with arbitrary
precision $E_W(f)$ from finitely many values of $f$, if we know
already $E_0(f)$ and have some information on the
speed with which this limit is approximated.

\vskip .3in

\section{The algorithm}
%now we will give an algorithm which allows us to estimate with arbitrary
%precision $E_W(f)$ from finitely many values of $f$, if we know
%already $E_0(f)$ and have some information on the
%speed awith which this limit is approximated.

 We assume that the function $f$ from $\N$ to $\R^{+}$
satisfies the conditions $(\cal C^*)$. We don't loose generality with 
this assumption, since 
if there exists an integer $n$ such that $f(n) < n+1$ we have $E_W(f) = 0$ and if not, it follows from Remark 7.3 from [MM17] that
we may always change the function $f$ by a function
 $\tilde { \tilde f}$ satisfying
conditions $(\cal C^*)$ and such that $E_W(\tilde { \tilde f}) = E_W(f)$.

\begin{theorem}\label{durdur}
 There is 
an algorithm which gives, starting from the function $f$ and $\ve \in ]0, 1[$, a quantity $h$ such that
$(1-\ve)h\le E_W(f) \le h$.
The quantity $h$ depends explicitely on $\ve$, $E_0(f)$, $N$, $f(1)$, ..., $f(N)$, for an integer
$N$ which depends explicitely on $\ve$, $E_0(f)$ and an integer $n_0$. larger than an explicit function 
of $\ve$ and $E_0(f)$ and
such that
$$\frac{\log f(n)}{n} < (1+\frac{E_0(f) \ve}{210(4+2E_0(f))})E_0(f) \quad\mbox{for  any}\quad
 n \in \llbracket n_0, 2n_0 - 1 \rrbracket.$$
\end{theorem}

We shall now give the algorithm and prove Theorem \ref {durdur}.
The funcyion $f$ is given and henceforth we omit to mention it in $E_0(f)$ and $E_W(f)$.\\

{\bf Description of the algorithm}

For  $\ve \in ]0,1[$ given, let 
\begin{equation} \label {delta}
\delta:=\frac{E_0 \ve}{105(4+2E_0)}<\frac{\ve}{210}
\end {equation}
 and 
 \begin{equation} \label {K}
 K:=\lceil \delta^{-1} \rceil +1.
 \end {equation}
We choose a  positive integer 
\begin{equation} \label {n_0}
n_0\geq K\vee \frac{4K^2}{420^3E_0}
\end{equation}
 such
that for any integer $n \ge n_0$
\begin{equation} \label {majoration log}
\frac{\log f(n)}{n} < (1+\frac{\delta}{2})E_0.
\end {equation}
In view of conditions $(\cal C^*)$, 
this last condition is
equivalent to $\frac{\log f(n)}{n} < (1+\frac{\delta}{2})E_0$ for any $n \in \llbracket n_0, 2n_0-1 \rrbracket$.
We choose intervals which will be so large that all the lengths of words we manipulate 
stay in one of them. Namely, for each non-negative integer $t $,
let $$n_{t+1}:= \exp(K((1+\delta)^2E_0n_t+E_0)).$$  We 
take $$N:=n_K.$$
We choose now a set $Y \subset A^N$ and
we define $$q_n(Y):=|L_n(Y)|$$ for  $n \in \llbracket 1, N\rrbracket$. We look at those $Y$
for which
\begin{equation} \label {majoration q_n}
 q_n(Y) \le f(n)
\end {equation}
for any $n \in \llbracket 1, N \rrbracket$
and choose one among them such that
$$\min_{1 \le n \le N}\frac{\log q_n(Y)}n$$
%\quad \mbox{ 
is maximum.
Henceforth we omit to mention $Y$ in the notation $q_n(Y)$.

\begin{proposition}\label{qn}
We have $$\min_{1 \le n \le N}\frac{\log q_n}n \ge E_W.$$
 \end{proposition}
 
\begin{proof}   It follows from Section 4.3 of [MM17] (see (4))
that there is $\hat w \in W(f)$ with $p_n(\hat w) \ge \exp(E_W n)$ for any positive integer $n$.
For such a word $\hat w$, let 
 $$X:=L_N(\hat w)\subset A^N.$$ We have, for each for  $n \in \llbracket 1, N \rrbracket$, $L_n(X)=L_n(\hat w)$ and
$f(n)\ge  |L_n(\hat w)|=p_n(\hat w)\ge \exp(E_W n)$. Thus $X$ is one of the possible $Y$and the result
follows from the maximality of $\min_{1 \le n \le N}\frac{\log q_n}n$.
\end{proof}

The next lemma shows that on one of the large intervals we have defined, the quantity $\frac{\log q_n}n$ will be almost constant:
\begin{lemma}\label{nr}
There exists a non-negative integer
 $r<K$, such that
$$\frac{\log q_{n_r}}{n_{r}} < (1+\delta) \frac{\log q_{n_{r+1}}}{q_{n_{r+1}}}.$$\end{lemma}
\begin{proof}
Otherwise we would have
\begin{equation} \label {iteration}
\frac{\log q_{n_0}}{n_0}\geq (1+\delta)^K \frac{\log q_{n_K}}{q_{n_K}}.
\end {equation}
As $K>\frac{1}{\delta}$, we have $(1+\delta)^K = e^{K \log(1 + \delta)} > e^{\frac 1 {\delta} \log(1 + \delta)} > \frac 94$ for $\delta < \frac 12$.
By Proposition \ref{qn}, we have $\frac{\log q_{n_K}}{n_K} \geq E_W$, so that (\ref {iteration}) would implies that
$\frac{\log q_{n_0}}{n_0}\geq \frac{9}{4}E_W$.
But 
it follows from (\ref {majoration q_n}) that $q_{n_0}\leq f(n_0)$ 
and from (\ref {majoration log}) that
$\frac{\log q_{n_0}}{n_0} < (1+\frac{\delta}{2})E_0 \le \frac 98 E_0$ for $\delta < \frac 14$.
Finally we would have $E_W \le \frac 12 E_0$
which would contradict Theorem \ref {1/2theorem}.\end{proof}

If we put $$h:=\frac{\log q_{n_{r}}}{n_{r}},$$
the next proposition follows immediately from
Proposition \ref{qn} :

\begin{proposition}\label{majoration}
We have $$h \ge E_W.$$
 \end{proposition}

We shall use the estimates given by the following lemma:
\begin{lemma}\label{he}
We have $$ \frac{E_0}{2}\leq h\leq E_0(1+\frac{\delta}{2})$$\end{lemma}

\begin{proof}
It follows from Proposition
\ref{qn} and Theorem  \ref {1/2theorem} that $h \ge E_W > \frac {E_0}{2}$.
On the other hand, it follows from (\ref {majoration q_n}) that
$q_{n_r}\leq f(n_r)$ and as $n_r>n_0$, it follows from (\ref {majoration log}) that  $h\leq E_0(1+\frac{\delta}{2})$.
\end{proof}

What remains to prove is the following proposition (which, understandably, does not use 
the maximality of $\min_{1 \le n \le N}\frac{\log q_n}n$).
\begin{proposition}\label{dur}
We have
$$(1-\ve)h\leq E_W.$$
\end{proposition}
\begin{proof}
Our strategy is to build a word $w$ such that, for any positive integer $n$,
$$\exp((1-\ve)hn)\leq  p_n(w)\leq f(n),$$ which gives 
the conclusion by definition of $E_W$. To build the word $w$, we shall define an integer $m$ and
build successive subsets of $L_m(Y)$.
We order any such a subset $Z$ (lexicographically for example)
and define $w(Z)$ by using a Champernowne-type construction: namely, if 
$Z=\{\be_1,\be_2,...,\be_t\}$, we build the infinite word 
$$w(Z):=\be_1\be_2\dots\be_t\be_1\be_1\be_1\be_2\be_1\be_3\dots\be_{t}\be_t\be_1
\be_1\be_1\dots\be_t\be_t\be_t\dots$$
made by concatenation of  all words in $Z$ followed by
the concatenations of all pairs of words of $Z$ followed by
the concatenations of all triples of words of $Z$, etc...
(see [Ch33] and [MS98] for statistical properties of Champernowne words).
  
The word $w(Z)$ will satisfy $\exp((1-\ve)hn)\leq  p_n(w(Z))$ 
 for any positive integer $n$ as soon as 
 $$|Z| \ge \exp((1- \ve)h m)$$ since, 
for every positive integer $k$, 
we will have at least $|Z|^k$ factors
of length $km$ in $w(Z)$. 

The successive (decreasing) subsets $Z$ of $L_m(Y)$ we build will all have cardinality at least 
$\exp((1- \ve)h m)$ and the words $w(Z)$ will satisfy $p_n(w(Z))\leq f(n)$ for  $n$ in an interval 
which will increase at each new set $Z$ we build and ultimately contains all the integers.

We begin by an estimate on $q_n$ using the value of $h$.

\begin{lemma}\label{cqn}
For any $n \in \llbracket 1, N \rrbracket$, we have $q_n\leq \exp (hn+hn_r).$
\end{lemma}
\begin{proof}
For any integer non-negative integer $n \le N$ we write $n=an_r+b$ with $a$ non-negative integer and $b\in \llbracket 0, n_r - 1 \rrbracket$.
 As we have  $L_{n}(Y)\subset L_{an_r}(Y) L_{b}(Y) \subset (L_{n_r}(Y))^a L_{b}(Y)$ and
$ q_{n_{r}}=\exp(hn_{r})$, we get $$q_n\leq q_{n_{r}}^aq_b=
\exp(ahn_r)q_b\leq \exp(hn)q_{n_r} = \exp(hn_{r})\exp (hn).$$
\end{proof}

The following lemma uses only properties of $f$, independently of the definition of $Y$.

\begin{lemma}\label{hatn} For any integer $n\geq n_0$, there exists $n'\in \llbracket n, (1+\delta)n \rrbracket$
 such that  $$f(n'+j) \ge \exp(\frac{E_0 j}{2})f(n')$$ for every positive integer $j$.
\end{lemma}
\begin{proof}
Otherwise there would exist $j_0$ such that $f(n+j_0)< \exp(\frac{E_0 j_0}{2})f(n)$, then
there would exist $j_1$ such that 
$f(n+j_0+j_1)< 
\exp(\frac{E_0 j_1}{2})f(n+j_0)<\exp(E_0 \frac{j_0+j_1}{2})f(n)$ and so on
until we are out of the interval. Thus we would get
 some integer $s>\delta n$ such that 
$f(n+s)<\exp(\frac{E_0s}{2})f(n)$, but then, by the choice of $n_0$, we would have
$f(n+s)
<\exp(E_0 \frac{s}{2}) \exp((1+\frac{\delta}{2})E_0n)$. 
This last quantity is smaller than $\exp(E_0 
(n+s))$, because  $s>\delta n$ implies that $n\frac{\delta}{2}+\frac{s}{2}<s$,
This  would contradict the definition of $E_0$. \end{proof}

We are now ready to begin our construction. Our first aim is to define two lengths of words, 
$\hat n$ and $m$, which will be in the interval $[n_r,n_{r+1}]$  but with $m$ much larger
than 
$\hat n$ and a set  $Z_1$ of words of length $m$ of the form $\gamma\theta$, for
words $\ga$ of length $\hat n$, such that the word $\gamma\theta\gamma$ is in $L_{m+\hat n}(Y)$.
Thus, for a while, we shall be interested in twin occurrences of words.

Let  $\hat n$  be the $n'$ of Lemma \ref{hatn} 
defined for $n=Kn_r$ (note that the fact that $\hat n$ satisfies
the conclusion of Lemma \ref{hatn} will not be used before Lemma \ref{cour} much later).
Let
$$\hat N:= \lceil \exp(\frac{E_0}{2})f(\hat n)\rceil$$  
and
$$Y_1:=L_{\hat N}(Y).$$ 

We know that $Kn_r\leq \hat n\leq (1+\delta)Kn_r$. The first inequality implies that $n_r < \delta \hat n$
and the second inequality implies (by the initial choice of the $n_r$) that
$\hat N<n_{r+1}$.
We write $n_{r+1}=a\hat N+b$ with $a$ positive integer and  $b \in \llbracket 0, \hat N-1 \rrbracket$ and we use
the defining property of $r$ in Lemma \ref{nr}, which
translates into 
$$q_{n_{r+1}}\geq \exp(\frac{h n_{r+1}}{1+\delta})=\exp(\frac{h(a\hat N+b)}{1+\delta}).$$
On the other hand, we have by Lemma \ref{cqn},
 $$q_{n_{r+1}}\le q_{\hat N}^a q_b \le q_{\hat N}^a \exp(hb+hn_r).$$
Hence we get
$$q_{\hat N}^a \ge 
\exp(\frac{ha\hat N+hb}{1+\delta}-h(b+n_{r}))$$ and, as $a\geq 1$, this implies $$q_{\hat N}
\ge \exp(h(\frac{\hat N}{1+\delta}-\frac{\delta b}{1+\delta}-n_{r})).$$
As $b<\hat N$ 
and $n_r<\delta \hat N$, we get
$$|Y_1|=q_{\hat N} > \exp((1-3\delta)h\hat N).$$

For the moment, we fix a  word $W$ in $Y_1$.
The word $W$ has $\hat N -\hat n+1$ factors of length $\hat n$ and we have 
$\hat N-\hat n+1>(1+\frac{E_0}{2})f(\hat n)>f(\hat n)$. There are at most $f(\hat n)$
distinct factors of length $\hat n$. We make the list of the $c\leq f(\hat n)$ different words occurring 
in $W$, the $j$-th one
appearing
$a_j$ times, with $\sum_{j=1}^c  (a_j) >\hat N -\hat n+1$. We look at pairs of occurrences of 
the same factor, beginning at two different positions $s<t$. We denote such a pair by $(s,t)$
 and say two such pairs $(s,t)$ and $(s',t')$ are distinct if $t\neq t'$. Thus there are at least
 $\sum_{j=1}^c (a_j-1)\geq \hat N-\hat n+1-f(\hat n)$ distinct pairs.
To each pair $(s,t)$ we associate the interval $[s,t+\hat n[$.
The union of these intervals contains at least
$\hat N-\hat n+1-f(\hat n)+\hat n-1=\hat N-f(\hat n)$ integer points. 

Now we use the following elementary 

\begin{lemma}\label{int}
Given a finite family of intervals $(I_j)_{1\leq j\leq d}$, there is
a subfamily of disjoint intervals $(I_j)_{j\in J}$ such that
$$\sum_{j\in J} \ab I_j \ab \geq \sum_{i=1}^d \ab I_j\ab.$$ \end{lemma}
\begin{proof} We number the $I_j$ by ascending order of their lowest elements. Let $\hat d$ be
the largest $j$ such that $I_j\supset I_d$. We can remove all the $I_j$ for $\hat d <j\leq d$, if
they exist.
Then if $I_j\cap I_{j+2}
\neq \emptyset$, $I_{j+1}$ must be included in $I_j \cup I_{j+2}\cup \ldots I_{\hat d}$
and we can remove $I_{j+1}$. Thus, after removing some intervals and renumbering, we can suppose all the
$I_j\cap I_{j+2}$ are empty. Then either the family of even-numbered intervals or the family
of odd-numbered intervals satisfies our requirements. \end{proof}

We apply Lemma \ref{int} to the above intervals $[s,t+\hat n[$, for the word $W$.
Thus we get some $\ell$ and 
$s_1<t_1<\ldots <s_\ell<t_\ell$, such that the same factor of $W$ occurs at positions $s_i$ and $t_i$ and
the sum of the lengths $\sum_{i=1}^{\ell} (t_i+\hat n-s_i)$ is at least
 $$\frac{\hat N-f(\hat n)}{2} \ge \frac{E_0}{4+2E_0}\hat N,$$
 because 
$\hat N\geq e^{\frac{E_0}{2}}f(\hat n)$ and 
$\frac{1-e^{-\frac{E_0}{2}}}{2}\geq \frac{\frac{E_0}{2}}{2(1+\frac{E_0}{2})}$.

Since $t_i+\hat n-s_i \ge \hat n$ for each $i \le \ell$, we have $\ell \le \frac{\hat N}{\hat n}$. 

Now, if we look at  all $W$ in $Y_1$, the number of possible choices for the pairs $(s_i,t_i), 
1 \le i \le \ell$ is 
at most $$\sum_{\ell =1}^{\frac{\hat N}{\hat n}}{\hat N \choose {2 \ell}} \le
\exp(\frac {4\hat N\log \hat n}{\hat n})$$ and this is smaller than $\exp(\delta \hat N)$
because $\hat n\ge n_0 K \ge K^2>\frac{1}{\delta^2}$, thus 
$$\frac{4\log\hat n}{\hat n}\leq 
8\delta^2\log\frac{1}{\delta}<\delta \frac{E_0}{2}\leq \delta h.$$

The cardinality of $Y_1$ is at least
$\exp((1-3\delta)h\hat N)$, thus we can find a subset  $Y_2\subset Y_1$ of at least 
$\exp((1-4\delta)h\hat N)$ 
elements of $Y_1$ which have the same choice of pairs $(s_i,t_i)$.

We define,  for  $(s, t) \in \llbracket 1, \hat N \rrbracket^2$ with $s<t$, the projections
$\pi_{s,t}\colon Y_2\to A^{t-s}$ 
 by $\pi_{s,t}(\be_1,\be_2,\dots,\be_{\hat N})=(\be_s,\be_{s+1},\dots,\be_{t-1})$.

Let 
\begin{equation} \label {epsilon}
\tilde \ve= \frac{\ve}{15}=\frac{7(4+2E_0)\delta}{E_0}>14 \delta.
\end{equation}

\begin{lemma}\label{pair}
There is a pair $(s_i,t_i)$  such
that $$|\pi_{s_i,t_i+\hat n}(Y_2)| \ge \exp((1-\tilde \ve)h \cdot (t_i+\hat n-s_i)).$$
\end{lemma}
\begin{proof}
Suppose 
by contradiction that for each $i \in \llbracket 1, \ell \rrbracket$ we
have $$|\pi_{s_i,t_i+\hat n}(Y_2)| < \exp((1-\tilde \ve)h \cdot (t_i+\hat n-s_i)).$$ 
The interval
$[1, \hat N[$ can be written as the union of the intervals $[s_i,t_i+\hat n[, i \in \llbracket 1, \ell \rrbracket$ with
at most $\ell+1$ holes. Let $M$ be the sum of the lengths of the holes, we have proved $M$ is at most
$(1-\frac{E_0}{4+2E_0})\hat N$.

By Lemma \ref{cqn}, an
upper bound for the number of possible sequences 
in these holes is
$$\exp((\ell+1)h n_{r}+hM) \le \exp(hM+2\delta h \hat N).$$ This would give
an upper estimate for the total number of words in $Y_2$ of the
order of $$\exp(h \hat N) \exp(-\tilde \ve h \frac{E_0}{4+2E_0} \hat N) \exp(2 \delta \hat N)\le$$
$$\exp(h \hat N) \exp(-7\delta h \hat N) \exp(2\delta h \hat N)=
\exp((1-5\delta)h \hat N),$$ which would contradict the lower estimate $\exp((1-4\delta)h \hat N)$.
\end{proof}

Now we fix a pair $(s_i,t_i)$ such that
$$|\pi_{s_i,t_i+\hat n}(Y_2)| \ge \exp((1-\tilde \ve)h(t_i+\hat n-s_i)).$$

For a 
word in $Y_2$, the sequence of its letters whose positions are
in the interval $[s_i,t_i+\hat n[$ is such that its last $\hat n$ 
letters coincide with its first $\hat n$ letters. Bounding $q_{t-i+\hat n -s_i}$ by Lemma \ref{cqn} and using $n_r<\delta\hat n$,  we get
$$|\pi_{s_i,t_i+\hat n}(Y_2)| \le
\exp(h(t_i+\hat n-s_i+\delta\hat n-\hat n)),$$
which  because of the above choice of the pair  implies 
$(1-\delta)\hat n \le \tilde \ve(t_i+\hat n-s_i)$ and 
$t_i+\hat n-s_i-\hat n>\frac 1 {2\tilde \ve} \hat n$. 

If we put $$m:=t_i-s_i,$$ 
we have 
\begin{equation} \label {m}
m>\frac {\hat n}{2\tilde \ve}.
\end{equation}
We shall need  the following upper bound.

\begin{lemma}\label{hmh} We have $
m<\exp(\frac{E_0}{2}\tilde\ve m).$
%\leq \exp(h\tilde\ve m).$
\end{lemma}
\begin{proof}
Let $\phi$ the function defined for any $x \in \R^+$ by $\phi (x) = \exp(\frac{E_0}{2}\tilde\ve  x)-x.$ 

The function $\phi$ is increasing on the interval $[\frac {K^2}{2\tilde \ve}, +\infty [$:
we have $\phi '(x) = \frac{E_0}{2}\tilde\ve \exp(\frac{E_0}{2}\tilde\ve x)-1$ and $\phi '(\frac{K^2}{2\tilde\ve}) > 0$
because it follows from (\ref{delta}), (\ref{K}) and (\ref{epsilon}) that $K^2 > (\frac{420}{E_0 \ve})^2 > \frac{420^2}{{E_0}^2 \ve} > \frac{8}{{E_0}^2\tilde \ve}$, so that
\footnote{For any $x \in \R^+$, we have $\exp(x)> x$.}
$$ \exp(\frac{E_0}{4}K^2) > \frac{E_0}{4}K^2 > \frac{2}{E_0\tilde\ve}.$$
It follows from (\ref {m}) and (\ref {n_0}) that $m> \frac {\hat n}{2\tilde \ve}>\frac {Kn_0}{2\tilde \ve}
> \frac {K^2}{2\tilde \ve}$, so that $$\phi (m) > \phi (\frac {K^2}{2\tilde \ve})$$
and it follows from (\ref{delta}), (\ref{K}) and (\ref{epsilon}) that ${E_0}^3K^4 > \frac{(420)^3}{\ve^3} \frac{14}{\tilde \ve} > \frac{(420)^3 14}{\tilde \ve} > \frac{(4)^3 3}{\tilde \ve}$, so that
\footnote{For any $x \in \R^+$, we have $\exp(x)> \frac{x^3}{6}$.}
$$\phi (\frac {K^2}{2\tilde \ve}) =  \exp(\frac{E_0}{4}K^2) - \frac {K^2}{2\tilde \ve} > \frac{{E_0}^3 K^6}{6 (4)^3} - \frac {K^2}{2\tilde \ve} > 0.$$

  \end{proof}
   
The set
$$Z_1:=\pi_{s_i,t_i}(Y_2)$$ 
is made with words of length $m$ of the type 
$\ga \theta$ for words $\gamma$ of length $\hat n$, such that the word
$\ga\theta \ga$ is in $\pi_{s_i,t_i+\hat n}(Y_2)$. 
Thus 
$$|Z_1|=|\pi_{s_i,t_i+\hat n}(Y_2)|\geq \exp((1-\tilde \ve)h (n+\hat n).$$ 

Then we  consider the prefixes of length $6 \tilde \ve m \ge 3\hat n$
of words of $Z_1$ and their suffixes of length $6 \tilde \ve m \ge 3\hat n$.
By  Lemma \ref{cqn}, 
and $n_r<\delta\hat n$,
there are at most $\exp(12\tilde \ve h m+2\delta h\hat n)$ such subwords and, by choosing those
which are more frequent, we define a new set $Z_2\subset Z_1$ in which all the words 
have the same prefix $\ga_1$ of length $6\tilde\ve hm$ and all the words 
have the same suffix $\ga_2$ of length $6\tilde\ve hm$, with
$|Z_2|\ge|Z_1|\exp(-12\tilde \ve h m-2\delta h\hat n)$ and $2\delta h\hat n \leq (1-\tilde\ve )\hat n$,
thus $$|Z_2|\geq \exp((1-13\tilde \ve)h m).$$ 

 As a consequence of the definition of $Z_2$, all words of $Z_2$ have the same 
prefix of length $\hat n$, which is  a prefix $\ga_0$ of $\ga_1$. 
As $Z_2$ is included in $Z_1$, any word of $Z_2$ 
is of the form $\gamma_0\theta$ and the word $\gamma_0\theta\gamma_0$ is in $L_{m+\hat n}(Y)$.

We can now reap a (small)
first benefit of all this construction: by using the above property of $\ga_0$, we can bound by below $f(n)$
the number of very short factors of $w(Z_2)$.

\begin{claim}\label{c1} We have $p_{w(Z_2)}(n)\leq f(n)$ for any $n \in \llbracket 1, \hat n+1 \rrbracket$.\end{claim}
\begin{proof}
For $1\leq n\leq \hat n+1$, a factor $x$ of length $n$ of $w(Z_2)$  either is
a factor of a word of $Z_2$ and this word is some $\ga_0\theta$,
or else is made with
a suffix of length $u \in \llbracket 1, n-1 \rrbracket$ of a 
word $\ga_0\theta$ of $Z_2$ concatenated with a prefix of length $n-u \in \llbracket 1, n-1 \rrbracket$ of
another
word $\ga_0\theta'$ of $Z_2$, thus 
$x$ is a factor of $\gamma_0\theta\gamma_0$. In both cases $x$
is a factor of a word in  $L_{m+\hat n}(Y)$, thus is in $L_n(Y)$. Thus our claim
is satisfied
as $|L_n(Y)|\leq q_n\leq f(n)$.
\end{proof}

Let us shrink again our set of words.
\begin{lemma}\label{cour}
For a given subset $Z$ of $Z_2$, there exists
$Z'\subset Z$, $$|Z'|\geq (1-\exp(-(j-1)\frac{E_0}{2}))^j|Z|,$$ such that
the total number of factors of length $\hat n+j$ of all words $\ga_0\theta\ga_0$ such that 
$\ga_0\theta$ is in $Z'$ is at most $f(\hat n+j)-j$.
\end{lemma}
\begin{proof}
 Let $w_1$, \ldots, $w_c$, with $c\leq f(\hat n + j)$, the factors of length
$\hat n+j$ of all words  $\ga_0\theta\ga_0$ such that 
$\ga_0\theta$ is in $Z$. If $c<f(\hat n + j)$, we add arbitrary 
words (we call them ghost factors) $w_d$, $c<d\leq f(\hat n + j)$ of length $\hat n + j$, to
make $f(\hat n+j)$ different words.
For such a word $\ga_0\theta\ga_0$, its number of factors of length $\hat n+j$ is at most
$m+\hat n -(\hat n+j)+1$.

The proportion of subsets $\{w_{i_1},\ldots w_{i_j}\}$ of $j$  words (among the possible
$f(\hat n+j)$ factors of length $\hat n+j$,  including ghost factors)
 such that no $w_{i_r}$ is a factor of $\ga_0\theta\ga_0$ is at least
$$\frac{\binom{f(\hat n+j)-(m -j+ 1)}{j}}{\binom{f(\hat n+j)}{j}}=
\frac{f(\hat n+j)-(m -j+1)}{f(\hat n+j)}\ldots \frac{f(\hat n+j)-m}{f(\hat n+j)-j+1}
$$ $$>(\frac{f(\hat n+j)-m}{f(\hat n+j)})^j>(1-\exp((j-1)\frac{E_0}{2}))^j$$
as  $f(\hat n + j) \ge \exp(j E_0 /2)f(\hat n)$, by choice of 
$\hat n$ after Lemma \ref{hatn} and $m \le \hat N - \hat n \le \exp(\frac{E_0}{2})f(\hat n)$.

Thus on average a subset
of $j$ factors $w_t$ intersects a proportion at most $1-(1-e^{(j-1)\frac{E_0}{2}})^j$ of the words
$\ga_0\theta\ga_0$ for $\ga_0\theta$ in 
$Z$. There are as many words $\ga_0\theta$ in $Z$ as 
corresponding words $\ga_0\theta\ga_0$.
Thus there exists a set of $j$ factors $w_t$ and a subset $Z'$ of $Z$ of cardinality
at least $ (1-e^{-(j-1)\frac{E_0}{2}})^j|Z|$ such that none of the $j$ factors $w_t$ is a factor
of a word $\ga_0\theta\ga_0$ for $\ga_0\theta$ in $Z'$.\end{proof}

We start from $Z_2$ and apply successively Lemma \ref{cour} from $j=2$ to $j=6\tilde \ve m$,
getting $6\tilde \ve m-1$ successive sets $Z'$.
At the end, we get a set $Z_3$
such that the total number of factors of length $\hat n+j$ of words $\ga_0\theta\ga_0$ for $\ga_0\theta$
in $Z_3$ 
is at most $f(\hat n+j)-j$ for $j=2,\ldots, 6\tilde\ve m$ and $\frac{|Z_3|}{|Z_2|}$ is
at least
 $$\prod_{2 \le j \le 6 \tilde \ve m-\hat n}(1-\exp(-(j-1)\frac{E_0}{2}))^j \ge 
\prod_{j \ge 2}(1-\exp(-(j-1)\frac{E_0}{2}))^j:=p_0.$$
We have $$\log p_0=
\sum_{j \ge 2}j \log(1-\exp(-(j-1)\frac{E_0}{2}))$$
$$ > \sum_{j \ge 2}\frac{-j\exp(-(j-1)\frac{E_0}{2})}{1-\exp(-(j-1)\frac{E_0}{2})}.\footnote{For any $x \in ]0, 1$[, we have $\log(1-x)>-\frac{x}{1-x}$.}$$

It follows that
$$\log p_0 > \frac {-1} {1-\exp(-\frac{E_0}{2})} \sum_{j \ge 2} j\exp(-(j-1) \frac{E_0}{2}) = \frac{-\exp(-\frac{E_0}{2})(2-\exp(-\frac{E_0}{2}))}{(1-\exp(-\frac{E_0}{2}))^3} \ge \frac{-1}{(1-\exp(-\frac{E_0}{2}))^3},$$
which implies \footnote{For any $x \in ]0, +\infty [$, we have $\frac{1}{1-\exp(-x)}<1+\frac{1}{x}$.}
$$p_0 \geq  \exp(-(1+\frac{2}{E_0})^3).$$
Now $(1+\frac{2}{E_0})^3$ is smaller than $\tilde\ve hm$ because $$h\geq \frac{E_0}{2}, \quad
\tilde\ve m\geq \frac{\hat n}{2}\geq \frac{Kn_0}{2}$$ and from (\ref {n_0}) $$K\frac{E_0}{4}n_0\geq \frac{K^3}{420^3}
\geq (1+\frac{2}{E_0})^3,$$ thus
 $$|Z_3|
\ge \exp((1-14\tilde \ve)h m).$$

We can now bound the number of short factors by using the  factors we have just deleted
and properties of the words $\ga_0$, $\ga_1$ and $\ga_2$.

\begin{claim}\label{c2} We have $p_{w(Z_3)}(n)\leq f(n)$ for any $n \in \llbracket 1, 6\tilde\ve m \rrbracket$.\end{claim}
\begin{proof} Claim \ref{c1} is still valid for $Z_3\subset Z_2$, so we look at a factor $x$ in
$w(Z_3)$ of length $\hat n +j$ with $j \in \llbracket 2, 6\tilde\ve m -\hat n \rrbracket$. 
If $x$ is a factor of some $\ga_0\theta\ga_0$ for $\ga_0\theta$ in $Z_3$, there are at most
$f(\hat n+j)-j$ possibilities for $j$.
 We look at those $x$ which are  not  a factor of such a $\ga_0\theta\ga_0$.
Then $x$
 is made with
a suffix of length $u \in \llbracket 1, \hat n+j-1 \rrbracket \subset \llbracket 1, 6\tilde\ve m \rrbracket$ of a 
word $\ga_0\theta$ of $Z_3$ concatenated with a prefix of length $\hat n+j-u \in \llbracket 1, 6\tilde\ve m  \rrbracket$
of
a word $\ga_0\theta`$ of $Z_3$ and we nust have $\hat n+j-u> \hat n$, otherwise $x$ would be a factor
of $\ga_0\theta\ga_0$.
As $Z_3\subset Z_2$,  the suffix is in $\ga_1$ and the prefix in $\ga_2$,
so the number of these possible $x$ is at most the number
of possible $u$, which range between $1$ and $j$. Thus the total  number of different
$x$ is at most $f(\hat n+j)-j+j$.
\end{proof}

We shrink our set again.

Let $n \in \llbracket  6 \tilde \ve m, m \rrbracket$.
In average a factor of length $n$ of a word in $Z_3$
occurs in at most $\frac{m|Z_3|}{f(n)}$ elements of $Z_3$ (we assume as above that
there are $f(n)$ possible words
  of size $n$, possibly by adding ghost factors). We consider the
$\frac{f(n)}{mn^2}$ factors of length $n$ which occur the least often. In total, these 
factors occur in at most $\frac{m|Z_3|}{f(n)}\frac{f(n)}{mn^2}=\frac{|Z_3|}
{n^2}$ elements of $Z_3$. We remove these words from $Z_3$, for any
$m \ge n > 6 \tilde \ve m$,  obtaining a set $Z_4$. We have removed
a proportion at most $1/n^2$ of $Z_3$ for each $n$ 
  with $m \ge n> 6 \tilde \ve m \ge 3 \hat n$, thus a total proportion 
  at most $\frac1{3 \hat n} \leq \frac{\delta}{3}< \frac{1}{630}$ of $Z_3$.
   This is smaller than
  $1-\exp(-\tilde \ve h m)$ by Lemma \ref{hmh}, thus  $$|Z_4|
\ge \exp((1-15\tilde \ve)h m).$$ 

We can now control medium length factors, using again the missing factors we have just
created and the words
$\ga_1$ and $\ga_2$ (but not $\ga_0$).
  
  \begin{claim}\label{c3} We have $p_{w (Z_4)}(n)\leq f(n)$ for any $n \in \llbracket 1, m \rrbracket$.\end{claim}
\begin{proof} Claim \ref{c2} is still valid for $Z_4\subset Z_3$. Let $6\tilde \ve m\leq n\leq m$ and 
$x$ a factor of length $n$ of $w(Z_4)$.
If $x$ is
a factor of a word in $Z_4$, by construction of $Z_4$
the number of different possible $x$ is
at most $f(n)-\frac{f(n)}{mn^2}$. 

If $x$ is not
a factor of a word in $Z_4$, 
then it
 is made with
a suffix of length $u \in \llbracket 1, n-1 \rrbracket$ of a 
word of $Z_4$ concatenated with a prefix of length $n-u$ of
another word of $Z_4$. 
If $u\geq 6\tilde\ve m$, let $u_1=u-6\tilde\ve m$, $u_2=n-u$.
 If $u<6\tilde\ve m$, let $u_1=0$,
$u_2=n-6\tilde\ve m$. Our $x$ is made with
a variable word of length $u_1$, concatenated with a factor of $\gamma_1\gamma_2$ which depends only on
$u$, $1\leq u\leq n-1$, concatenated with a variable word of length $u_2$. The $u_i$ depend also only
on $u$ and, 
for $u_i$ fixed, the
number of possible words of length $u_i$ is at most $q_{u_i}$.
 By  Lemma \ref{cqn}, we get
$$q_{u_1}q_{u_2}\leq \exp(h(2n_r+u_1+u_2))\leq  \exp(h(2n_r+n-6 \tilde \ve m)).$$

Thus by Lemma \ref{he} the number of possible  $x$ which are not factors of words in $Z_4$ is at  most
$$n \exp(h(2n_r+n-6 \tilde \ve m )) \le
 m \exp(h(2n_r+n-6 \tilde \ve m ))$$
$$ \le m  \exp(2n_rh-6h \tilde \ve m) \exp(E_0(1+\frac{\delta}{2}) n)$$ 
$$<m 
 \exp(2n_rh+(E_0\frac{\delta}{2}-6h \tilde \ve) m) \exp(E_0 n).$$
 As $n_r<\delta m$, this number is strictly smaller than
 $$m \exp(2hm\delta+(E_0\frac{\delta}{2}-6h \tilde \ve)
 m) \exp(E_0 n)
 <m \exp(-3h \tilde \ve m) \exp(E_0 n)$$
 because
 $\delta \frac{E_0}{2}<h\tilde\ve$  by Lemma \ref{he} and $\delta<\frac {\tilde \ve}{14}$ by (\ref {epsilon}).
  By Lemma \ref{hmh},  our last estimate on the number of possible $x$ is at most
   $$\frac{\exp(E_0 n)}{m^3}\le\frac{f(n)}{{m}^3} \le \frac{f(n)}{m n^2}$$
 and our claim is proved.\end{proof}

Finally we put $Z_5=Z_4$ if $|Z_4|\leq \exp((1-4\tilde\ve)hm)$, otherwise we take for $Z_5$ any subset of 
$Z_4$ with
$\lceil \exp((1-4\tilde\ve)hm)\rceil$ elements. 
In both cases we have 
$$|Z_5|
\ge \exp((1- \ve)h m).$$ 

For the long factors, we use mainly the fact that there are many missing factors of length $m$,
but we need also some help from $\ga_1$ and $\ga_2$
  
  \begin{claim}\label{c4} We have $p_{w(Z_5)}(n)\leq f(n)$ for any $n$.\end{claim}
\begin{proof} Claim \ref{c3} is still valid for $Z_5\subset Z_4$. Let $x$ be a factor of $w(Z_5)$
 of length $n>m$, with $n=Qm+u$,  $0\le u<m$,  $Q\geq 1$ and thus $$Qm\geq n/2.$$
 The word $x$ is made with a suffix of length $u_1$ of a word of $Z_5$, concatenated with $Q'$ words of $Z_5$
 concatenated with a 
 prefix of length $u_2$ of a word of $Z_5$. According to the value of $u_1$, 
 there are two possibilities: 
 
- first case
  $Q'=Q$ and $u_1+u_2=u$ and this occurs for $m_1$ possible values of $u_1$;
  
 - second case
 $Q'=Q-1$ and
$u_1+u_2=m+u$ and this occurs for $m-m_1$ values of $u_1$.  
 
In the first case 
we bound $q_{u_1}q_{u_2}$ by 
Lemma \ref{cqn} and  the number of possible $x$ by 
$$p_1=m_1\exp(hu)\exp((1-4\tilde\ve)hmQ+2hn_r).$$
Thus $p_1=m_1\exp(hn)\exp(-4\tilde\ve hmQ+2hn_r)$, where $n_r$ is at most $\delta m$ and $Qm$ is at least $n/2$, thus
 $$p_1\leq m_1\exp(hn)\exp(2hm\delta)\exp(-2\tilde\ve hn).$$

 In the  second case, either the initial suffix of length $u_1$
or the final prefix of length $u_2$ contains one of the fixed words
$\gamma_1$ or $\gamma_2$ of length 
$6\tilde\ve m$ and, using again Lemma \ref{cqn}, we bound
the number of possible $x$ by 
$$p_2=(m-m_1)\exp(h(m+u)-6h\tilde\ve m)\exp((1-4\tilde\ve)h m(Q-1)+2hn_r).$$
We have $h(m+u)+hm(Q-1)=hn$ and use $n_r<\delta m$, thus
$p_2$ is at most
$$(m-m_1)\exp((-6\tilde\ve+2\delta) hm)\exp(hn)\exp(-4\tilde\ve h(Q-1)m)$$
$$\le (m-m_1)\exp((-2\tilde\ve+2\delta) hm)\exp(hn)\exp(-4\tilde\ve hQm)$$
$$\le (m-m_1)\exp((-2\tilde\ve+2\delta) hm)\exp(hn)\exp(-2\tilde\ve hn)$$
$$\le (m-m_1)\exp((2\delta hm)\exp(hn)\exp(-2\tilde\ve hn).$$

Finally, we have $$p_n(w(Z_5)) \le m\exp(hn)\exp(2hm\delta)\exp(-2\tilde\ve hn).$$

By Lemma \ref{hmh} 
 and Lemma \ref{he} we have $m \le \exp(\tilde\ve hm) \le \exp(\tilde\ve hn)$ (because $m \le n$). Thus 
$$p_n(w(Z_5))\leq \exp(E_0n)\exp(2hm\delta)\exp(-\tilde\ve hn)\exp(nE_0\frac{\delta}{2})$$
$$\leq \exp(E_0n)\exp(h(2m+n)\delta)\exp(-\tilde\ve hn)$$
% (by the lower bound in Lemma \ref{he})
$$\leq \exp(E_0n)\exp(3hn\delta)\exp(-\tilde\ve hn).$$ 
%(because $m\leq n$) and as 
As we have $\tilde \ve>3\delta$ by (\ref {epsilon})
we get $p_n(w(Z_5)) \leq \exp(E_0n)\leq f(n)$.
\end{proof}

In view of the considerations at the beginning of the proof of Proposition \ref{dur}, Claim \ref{c4} 
completes the proof of Proposition \ref{dur} and thus of
Theorem \ref{durdur}.\end{proof}

\end{document}